 \documentclass[a4paper,11pt]{article}

\usepackage{amsthm,amsmath,hyperref,geometry,color,stmaryrd,bbm}
\usepackage[utf8]{inputenc}
\usepackage{amssymb}
\usepackage[english]{babel}
\usepackage{graphicx}
\usepackage{amsfonts,amssymb}
\usepackage{verbatim}
\usepackage{enumitem}

\newtheorem{thm}{Theorem}
\newtheorem{prop}[thm]{Proposition}

\newtheorem{lem}[thm]{Lemma}
\newtheorem{cor}[thm]{Corollary}

\newtheorem*{defi*}{Definition}
\newtheorem{assu}{Assumption}

\newcommand{\po}{\left(}
\newcommand{\pf}{\right)}

\newcommand{\cco}{\llbracket}
\newcommand{\ccf}{\rrbracket}
\newcommand{\R}{\mathbb R}
\newcommand{\N}{\mathbb N}

\newcommand{\dd}{\mathrm{d}}

\newcommand{\na}{\nabla}

\newcommand\ent[2]{\mathcal H \left(\left.#1\ \right|\ #2\right)}
\newcommand{\1}{\mathbbm{1}}

\title{Uniform long-time and propagation of chaos estimates for mean field kinetic particles in non-convex landscapes}
 \author{Arnaud Guillin, Pierre Monmarch\'e}

\begin{document}

\maketitle

\begin{abstract}
Combining the results of \cite{MonmarcheVFP} and \cite{GuillinWuZhang}, the trend to equilibrium in large time is studied for a large particle system associated to a Vlasov-Fokker-Planck equation. Under some conditions (that allow non-convex confining potentials) the convergence rate is proven to be independent from the number of particles. From this are derived uniform in time propagation of chaos estimates and an exponentially fast convergence for the nonlinear equation itself. 
\end{abstract}

\section{Introduction}

This work is devoted to the study of the long-time convergence of the solutions of the  Vlasov-Fokker-Planck equation, governing the evolution of the density of interact- ing and diffusive matter in the space of positions and velocities, and of the associated system of interacting particles, and of the convergence of the latter to the former as the number of particles increases. More precisely, following the notations of \cite{MonmarcheVFP}, the Vlasov-Fokker-Planck equation is
\begin{eqnarray}\label{EqVlasovFP}
\partial_t m_t +y \cdot \na_x m_t\   &  =& \na_y \cdot \po \frac{\sigma^2}{2} \na_y m_t + \po \int_{\R^d} \na_x U\po x, x' \pf   m_t(x',y') \dd x' \dd y' + \gamma y\pf m_t\pf
\end{eqnarray}
where $m_t(x,y)$ is a density at time $t$ of particles at point $x\in \R^d$ with velocity $y\in\R^d$, $d\in\mathbb N_*$, $\sigma,\gamma>0$, $\na$ and $\na\cdot$ stand for the gradient and divergence operators and the potential $U$ is a  $\mathcal C^1$ function from $\R^{2d}$  to $\R$ with $U(x,x')=U(x',x)$ for all $x,x'\in\R^d$. For $N\in\mathbb N_*$,  
the associated system of $N$ interacting particles is the Markov process  $Z_N=(X_i,Y_i)_{i\in\cco 1,N\ccf}$ on $\R^{2dN}$ that solves the stochastic differential equation
\begin{eqnarray}\label{EqSystemparticul}
\forall i\in \cco 1,N\ccf & &\left\{\begin{array}{rcl}
\dd X_i & = & Y_i \dd t\\
\dd Y_i & = & - \gamma Y_i \dd t - \po \frac{1}{N}\underset{j=1}{\overset{N}\sum}  \na_x U \po X_i, X_j\pf \pf\dd t + \sigma \dd B_i 
\end{array} \right. 
\end{eqnarray} 
with the initial conditions $(X_i(0),Y_i(0))$ being i.i.d. random variables of law $m_0$, independent from the standard Brownian motion $B=(B_1,\dots,B_N)$ on $\R^{dN}$. As $N\rightarrow \infty$, one expects that the particles are approximately independent so that a Law of Large Number holds and the empirical law 
\[M_t^{N}=\frac1N \sum_{i\in \cco 1,N\ccf} \delta_{(X_i,Y_i)}\,,\]
which is a random probability measure on $\R^{2d}$, is close to the common law of the $(X_i,Y_i)$'s, whose evolution in time should thus approximately follow Equation \eqref{EqVlasovFP}. This is the so-called propagation of chaos phenomenon, as introduced by \cite{Kac} and further developped by \cite{Sznitman}. Rigorous statements  are provided below.

 The long-time behaviour of $m_t$ has been studied in various settings. Convergence to equilibrium without quantitative speed is addressed in \cite{TugautDuong}. Decomposing the potential $U(x,x') = V(x)+V(x')+W(x,x')$ where $V$ and $W$ are respectively called the confinement and interaction potentials, exponentially fast long-time convergence is established by perturbation of the linear case in \cite{Carrillo,HerauThomann}  when the interaction is sufficiently small. Such a quantitative result is also proven in \cite{BolleyGuillinMalrieu} when the potential is  close to a quadratic function, and in \cite{MonmarcheVFP} when $x\mapsto U(x,x')$ is stricly convex for all $x'$. Similarly to \cite{MonmarcheVFP}, in the present work, we will obtain the long-time convergence of $m_t$ from the long-time convergence of $m_t^{(N)}$ the law of $Z_N(t)$.
 
Indeed, remark that $Z_N$ is a classical Langevin diffusion, for which relaxation toward equilibrium has been addressed, under various assumptions on the potential,  in a broad number of works and with various techniques like Meyn-Tweedie or coupling probabilistic approaches \cite{Talay,EberleGuillinZimmer} or hypocoercive modified entropy methods \cite{Talay,Villani,DMS,CGMZ}, see also \cite{BaudoinGordinaHerzog} and within for more recent references. With respect to all this litterature, the specificities of \cite{MonmarcheVFP} that are relevant in the present mean-field framework are twofold: first, the long-time convergence has to be quantified in relative entropy (total variation distance or $L^2$ or $H^1$ norms would not be suitable for the limit $N\rightarrow \infty$) and, second, the convergence rate should be independent from $N$ (which is not the case for example in \cite{BaudoinGordinaHerzog}). From this, combined with crude propagation of chaos  estimates, long-time convergence is obtained in \cite{MonmarcheVFP} for the  non-linear limit equation \eqref{EqVlasovFP}, together with uniform in  time propagation of chaos estimates. It turns out that  there is mainly one step in \cite{MonmarcheVFP} where the convexity of the potential is crucially used, which is the proof that $m_\infty^{(N)}$  the invariant measure of $Z_N$ satisfies a log-Sobolev inequality with constant independent from $N$. However, in the recent \cite{GuillinWuZhang}, such a uniform inequality is proven for the invariant measure of the overdamped version of the system \eqref{EqSystemparticul}, which is exactly the $x$-marginal of $m_\infty^{(N)}$, under assumptions that allows non-convex potentiels but with superquadratic confinement. Since log-Sobolev inequalities are stable under tensorization and since such an inequality is clearly satisfied by the $y$-marginal of $m_\infty^{(N)}$, which is a Gaussian law, we are in position to extend the results of \cite{MonmarcheVFP} to a much broader class of potentials.  

The plan of the paper is quite simple. Section 2 will present the results and comparisons with existing results, while proofs are provided in Section 3. We now detail these results.

\section{Results}

For $N\in\N_*$, denoting $\beta: = 2\gamma/\sigma^2$,  we consider the Gibbs measure with Hamiltonian
\[H_N(x,y) \ = \ \beta\po \frac{|y|^2}2 + U_N(x)\pf\,, \qquad \text{where } \qquad U_N(x) \ = \ \frac1{2N} \sum_{i=1}^N \sum_{j=1}^N U(x_i,x_j)\,, \]
namely the measure on $\R^{2dN}$  with Lebesgue density
\begin{eqnarray}\label{Eq:MesureInvar}
m_\infty^{(N)}(x,y) & =  &  \mathcal Z_N^{-1} \exp\po -  H_N(x,y)  \pf\,,\qquad \mathcal Z_N  := \int_{\R^{2dN}} \exp\po -  H_N(x,y)  \pf \dd x \dd y \,.
\end{eqnarray}
In all the paper we denote identically a probability density and the corresponding probability measure.

\begin{assu}\label{Hyp:concret}
The potential $U$ is given by $U(x,x') = V(x)+V(x')+W(x,x')$ where $V\in\mathcal C^\infty( \R^d)$ and $W\in\mathcal C^\infty( \R^d\times\R^d)$  with all their derivatives of order larger than 2 bounded. There exist $c_U>0$, $c_U',c_W',R\geqslant 0$ and $c_W\in\R$  such that for all $x,y,z\in\R^d$,
\begin{eqnarray}\label{Eq:hypV}
\po \na V(x) - \na V(y) \pf \cdot   (x-y) & \geqslant & c_V |x-y|^2 - c_V' |x-y| \1_{\{|x-y|\leqslant R\}}\\
\po \na_x W(x,z) - \na_y W(y,z) \pf \cdot   (x-y) & \geqslant & c_W |x-y|^2 - c_W' |x-y| \1_{\{|x-y|\leqslant R\}}\label{Eq:hypW}
\end{eqnarray}
 Moreover, $U$ is the sum of a strictly convex function and of a bounded function, $W$ is lower bounded, $c_V+c_W > \|\na^2_{x,x'} W\|_\infty$ and  $\beta <\beta_0$ where
 \[\beta_0 \  := \  \frac{4}{(c_V'+c_W')R} \ln\po \frac{c_V+c_W}{\|\na^2_{x,x'} W\|_\infty}\pf \qquad  (:=\ +\infty\text{ if }(c_V'+c_W')R=0).\]
\end{assu}

Remark that Assumption \ref{Hyp:concret} discards singular potentials such as considered in \cite{BaudoinGordinaHerzog}. Indeed, we focus here on the question of having uniform estimates (in $t$ when $N\rightarrow \infty$ or in $N$ when $t\rightarrow \infty$) in non-convex cases, which is already interesting and new in cases where $U$ is smooth with bounded derivatives.

 We say a probability measure $\mu$ satisfies a log-Sobolev inequality with constant $\eta>0$ if
\begin{eqnarray}\label{EqLogSob}
\forall f>0\ s.t.\ \int f\dd \mu = 1,\hspace{40pt} \int f\ln f\dd \mu  & \leq & \eta \int  \frac{|\na f|^2}{f} \dd \mu.
\end{eqnarray}

\begin{prop}\label{Prop-logSob}
Under Assumption \ref{Hyp:concret}, there exists $\eta>0$ such that for all $N\in\N_*$, $\mathcal Z_N <+\infty$ and $m_\infty^{(N)}$ satisfies a log-Sobolev inequality with constant $\eta$. 
\end{prop}

Note that logarithmic Sobolev inequalities have direct consequences that may be useful beyond the convergence to equilibrium we look at in this paper. For example it entails uniform in the number of particles Gaussian concentration inequalities for the measure $m^{(N)}_\infty$. Another important consequence of a logarithmic Sobolev inequaliy is that it implies a Talagrand inequality. It will enable us to pass from entropic convergence to equilibrium to Wasserstein convergence to equilibrium. Let us detail this.

For $\mu$ and $\nu$ two probability laws on some Polish space $E$, we write
\[\ent{\nu}{\mu} = \left\{ \begin{array}{ll}
\int_E \ln \po \frac{\dd \nu}{\dd \mu}\pf \dd \nu & \text{if }\nu \ll \mu\\
+\infty & \text{else}
\end{array}\right.\]
the relative entropy of $\nu$ with respect to $\mu$ and
\begin{eqnarray*}
\mathcal W_2 \po \mu,\nu\pf &=&  \inf_{\pi\in\Gamma(\mu,\nu)} \left\{\sqrt{\mathbb E\po|A_1-A_2|^2\pf},\ Law(A_1,A_2)= \pi \right\}
\end{eqnarray*}
their $\mathcal W_2$-Wasserstein distance, where the infimum is taken over the set $\Gamma(\mu,\nu)$ of transference plan between $\mu$ and $\nu$, namely the set of probability laws on $E\times E$ with marginals $\mu$ and $\nu$. Recall that the set $\mathcal P_2(\R^d)$ of probability measures on $\R^d$ that have a finite second moment, endowed with the distance $\mathcal W_2$, is complete. Similarly, denote
\begin{eqnarray*}
\| \mu-\nu\|_{TV} &=&  \inf_{\pi\in\Gamma(\mu,\nu)} \left\{ \mathbb P\po A_1\neq A_2\pf ,\ Law(A_1,A_2)= \pi \right\}
\end{eqnarray*}
the total variation norm of $\mu-\nu$. Recall Pinsker's Inequality
\[\|\mu-\nu\|_{TV}^2 \ \leqslant \ 2 \mathcal H \po \mu\ |\ \nu\pf\]
for all  $\mu,\nu\in \mathcal P(E)$, and Talagrand's $T_2$ Inequality
\[\mathcal W_2^2(\mu,\nu) \ \leqslant \ \eta \mathcal H \po \mu\ |\ \nu\pf\]
that holds for all $\mu\in \mathcal P(E)$ if $\nu$ satisfies a log-Sobolev inequality with constant $\eta$, see \cite{OV}.

Under Assumption \ref{Hyp:concret}, \eqref{EqSystemparticul} admits a strong solution  $Z_N=\po (X_i,Y_i)\pf_{i\in\cco 1,N\ccf}$  for any initial condition  (see \cite{Meleard96}). Denote $m_t^{(N)}$ the law of $Z_N(t)$.

\begin{thm}\label{TheoLine}
Under Assumption \ref{Hyp:concret}, there exist $C>1,\chi>0$ that depend only on $U,\gamma,\sigma$ such that for all $N\in\N_*$, $t\geq 0$ and all initial condition $m_0^{(N)}\in \mathcal P(\R^{2dN})$ 
\begin{eqnarray}
\ent{m_t^{(N)}}{m_\infty^{(N)}} & \leqslant & Ce^{-\chi t} \ent{m_0^{(N)}}{m_\infty^{(N)}}\label{Eq:EntropieLine}\\
\ent{m_t^{(N)}}{m_\infty^{(N)}} & \leqslant & \frac{C}{(1\wedge t)^3} \mathcal W_2^2 \po m_0^{(N)},m_\infty^{(N)}\pf\label{Eq:RegulEntropie}\\
 \mathcal W_2 \po m_t^{(N)},m_\infty^{(N)}\pf & \leqslant & Ce^{-\chi t}  \mathcal W_2 \po m_0^{(N)},m_\infty^{(N)}\pf \label{Eq:W2Line}  \,.
\end{eqnarray}
\end{thm}
Remark that the fact $C>1$ is of course necessary here. If not then \eqref{Eq:EntropieLine} would imply back a logarithmic Sobolev inequality for $m^{(N)}_\infty$ with the Dirichlet form given by the dynamic  \eqref{EqSystemparticul}, which is false since this Dirichlet form is degenerate. Such results thus being not coercive are named hypocoercive. Note that under weaker conditions, such a result was given in $L^2/H^1$ in \cite{GLWZ} also independent of the number of particles.


To study the mean-field equation \eqref{EqVlasovFP}, following the notations of \cite{GuillinWuZhang}, we consider $\alpha$ the probability measure with Lebesgue density proportional to $\exp(-V(x)-|y|^2/2)$ and denote
\[E_f(\nu) \ = \ \mathcal H(\nu\ |\ \alpha) + \frac12\int W(x,x') \nu(\dd x)\nu(\dd x')\]
the so-called free energy of any $\nu \in\mathcal P(\R^{2d})$ and
\[\mathcal H_{W}(\nu) \ =  \ E_f(\nu) - \min_{\mu\in\mathcal P(\R^d)} E_f(\mu)\]
the corresponding mean-field entropy.

\begin{thm}\label{TheoNonLine}
Under Assumption \ref{Hyp:concret}, $E_f$ admits a unique minimizer $m_\infty\in\mathcal P(\R^d)$. Moreover, there exist $C,\chi>0$ that depend only on $U,\gamma,\sigma$ such that for all $t\geq 0$ and all initial condition $m_0\in\mathcal P(\R^{2d})$,
\begin{eqnarray}
\mathcal H_W(m_t) & \leqslant & Ce^{-\chi t} \mathcal H_W(m_0)\label{Eq:EntropieNonLin} \\
\mathcal H_W(m_t)  & \leqslant & \frac{C}{(1\wedge t)^3} \mathcal W_2^2 \po m_t ,m_\infty \pf \label{Eq:RegulNonLin}\\
  \mathcal W_2 \po m_t,m_\infty\pf & \leqslant & Ce^{-\chi t}  \mathcal W_2 \po m_0 ,m_\infty \pf \,.
\end{eqnarray}
\end{thm}

Remark that $m_\infty$ is necessarilly an equilibrium of \eqref{EqVlasovFP}, and thus it solves
\[m_\infty(x,y) \ \propto \ \exp\po - \beta \po V(x) + \frac12|y|^2 +  \int_{\R^d} W(x,x') m_\infty(x',y')\pf\pf\,.\]
In particular, the mean-field entropy $\mathcal H_W(\nu)$ differs from $\mathcal H(\nu|m_\infty)$ since, up to an additive constant, the first one is 
\[\int_{\R^{2d}} \nu \ln \nu + \int_{\R^{2d}} V \nu + \frac12 \int_{\R^{2d}\times\R^{2d}} W \nu \otimes \nu\]
while, up to an additive constant, the second one is 
\[\int_{\R^{2d}} \nu \ln \nu + \int_{\R^{2d}} V \nu + \int_{\R^{2d}\times\R^{2d}} W \nu \otimes m_\infty\,,\]
i.e. is the linearization  of the first one at $\nu=m_\infty$.

\bigskip

What is available in practice is the empirical distribution $M_t^N$ for finite 
$t\geqslant 0$ and $N\in\N_*$.

\begin{cor}\label{CorPW2}
Under Assumption \ref{Hyp:concret}, there exist $\chi>0$ that depends only on $U,\gamma,\sigma$ such that for  all initial condition $m_0^{(N)}=m_0^{\otimes N}$ with $m_0\in\mathcal P_2(\R^{2d})$, there exists $K>0$ such that for all $N\in\N_*$ and $t\geq 0$ ,
\begin{eqnarray*}
\mathbb E\po \mathcal W_2^2\po M_t^N , m_\infty \pf \pf & \leqslant &  K\po e^{-\chi t} +a(N) \pf
\end{eqnarray*}
where 
\[a(N) \ = \ \left\{ \begin{array}{ll}
N^{-1/2} & \text{if }d=1\\
 \ln(1+N) N^{-1/2} & \text{if } d=2\\
N^{-2/d} & \text{if } d\geqslant 3\,.
\end{array}\right.\]
\end{cor}

As shown in \cite[Proposition 2.1]{BolleyGuillinVillani}, such a result yields confidence intervals with respect to the uniform metric for a numerical approximation of $m_\infty$ by $M_t^N \ast \xi$ where $\xi$ is a smooth kernel. It would   also be interesting in order to get concentration inequalities independent of the number of particles for additive functionals of the trajectories of the particles.

\bigskip

Finally, we consider the limit $N\rightarrow +\infty$. For $n\in\cco 1,N\ccf$, denote $m_t^{(n,N)}$ the law of $\po (X_1,Y_1),\dots,(X_n,Y_n)\pf$.

\begin{cor}\label{CorChaosUniforme}
Under Assumption \ref{Hyp:concret}, there exists $\kappa>0$ that depend only on $U,\gamma,\sigma$ such that for  all initial condition $m_0^{(N)}=m_0^{\otimes N}$ with $m_0\in\mathcal P_2(\R^{2d})$, there exists $K>0$ such that for all $N\in\N_*$, $n\in\cco 1,N\ccf$  and $t\geq 0$ ,
\begin{eqnarray*}
\mathcal{W}_2\po m_t^{(n,N)},m_t^{\otimes n} \pf & \leq & \frac{K\sqrt{n}}{N^\kappa} \\
\|m_t^{(n,N)}-m_t^{\otimes n}\|_{TV} & \leqslant & \frac{K\sqrt{n}}{N^\kappa}\,.
\end{eqnarray*}
\end{cor}
Note that our approch to prove such uniform (in time) propagation of chaos do not lead to an optimal exponent $\kappa$ (which is 1/2) for the Wasserstein distance as seen in some more constrained example in \cite{BolleyGuillinMalrieu}. It would be interesting to consider a direct coupling approach to prove this result with sharp speed, as in \cite{DEGZ}.

\section{Proofs}

\subsection{Uniform log-Sobolev inequalities}\label{Subsec:proof_logsob}

\begin{lem}\label{Lem:Uquadra}
Under Assumption \ref{Hyp:concret}, there exist $\alpha_1,\alpha_2,\alpha_3>0$ such that for all $x,x'\in \R^d$,
\[\alpha_1 \po |x|^2 + |x'|^2 \pf - \alpha_3 \ \leqslant \  U(x,x') \ \leqslant  \ \alpha_2 \po |x|^2 + |x'|^2 \pf + \alpha_3\,.\]
\[|\na_x U(x,x')| \ \leqslant \ \alpha_2 \po |x|+|x'|\pf + \alpha_3\,.\]
\end{lem}
\begin{proof}
This is a straightforward consequence of the uniform bound on $\na^2 U$ and on the fact $U$ is the sum of a strictly convex and of a bounded function.
\end{proof}

From Lemma \ref{Lem:Uquadra}, we get that $\mathcal Z_N<+\infty$ for all $N\in\N_*$, which is the first claim of Proposition~\ref{Prop-logSob}. As we now explain, the second claim, i.e. the uniform log-Sobolev inequalities for $m_\infty^{(N)}$, $N\in\N_*$, follows   from \cite[Theorem 8]{GuillinWuZhang} (itself based on \cite[Theorem 0.1]{Zegarlinski}).  Denote 
\[\pi_\infty^{(N)}(x) \ = \ \int_{\R^{dN}} m_\infty(x,y) \dd y \ = \ \widetilde{\mathcal Z_N}^{-1} e^{-\beta U_N(x)}\,, \qquad\widetilde{\mathcal Z_N}\ = \ \int_{\R^{dN}} e^{-\beta U_N(x)} \dd x\]
the $x$-marginal of $m_\infty^{(N)}$.  A straightforward consequence of the log-Sobolev inequality for the Gaussian law and of the tensorization property of the log-Sobolev inequalities (see for example \cite{BakryGentilLedoux}) is the following:

\begin{lem}\label{Lem:tensori}
Suppose that $\pi_\infty^{(N)}$ satisfies a log-Sobolev inequality for some constant $\eta_N$. Then $m_\infty^{(N)}$ satisfies a log-Sobolev inequality with constant $\max\po \eta_N,\beta \pf$.
\end{lem}

The study is thus reduced to $\pi_\infty^{(N)}$, which is precisely the topic of \cite{GuillinWuZhang}. We now introduced the framework of the latter. As a first step, without loss of generality we suppose that $\beta=1$. 

\begin{assu}\label{Hyp:GWZ-1}
The potential $U$ is given by $U(x,y) = V(x)+V(y)+W(x,y)$ where
\begin{enumerate}
\item The confinement potential $V\in\mathcal C^2( \R^d)$, its Hessian matrix is bounded from below and there are two positive constants
$c_1 , c_2$ such that $x \cdot \nabla V (x) \geqslant  c_1 |x|^2 - c_2$ for all $ x \in\R^d$.
\item The interaction potential $W\in\mathcal C^2( \R^d\times\R^d)$, its Hessian matrix is bounded and
\[\int_{\R^{2d}} e^{-[V (x) + V (y) + \lambda W (x, y)]} \dd x\dd y \ <\  +\infty\,,\qquad \forall \lambda>0\,.\]
\end{enumerate}
\end{assu}

\begin{assu}[Zegarlinski’s condition]\label{Hyp:GWZ-2}
Denoting
\[b_0(r) \ = \ \sup_{x,y,z\in\R^d:|x-y|=r} \po - \frac{x-y}{|x-y|} \cdot \po \na_x U(x,z)-\na_y U(y,z)\pf \pf\]
 for $r> 0$, then
 \[c_{L} \ := \ \frac14 \int_0^\infty \exp\po \frac14 \int_0^s b(r) \dd r\pf  s \dd s \ < \ +\infty\,. \]
Moreover,
\begin{eqnarray}\label{Eq:gamma0}
\gamma_0 & := & c_{L} \sup_{x,y\in\R^d,|z|=1}  |\na_{x,y}^2 U(x,y) z| \ < \ 1\,.
\end{eqnarray}
\end{assu}

\begin{assu}[Uniform conditional log-Sobolev inequality]\label{Hyp:GWZ-3}
There exist $\rho>0$ such that for all $N\in\N_*$ and all $x_{\neq1}=(x_2,\dots,x_N)\in\R^{d(N-1)}$, the conditional law $\pi_{\infty,x_{\neq1}}^{(N)}$ on $\R^d$ with density proportional to $x_1\mapsto \pi_\infty^{(N)}(x)$ satisfies a log-Sobolev inequality with constant $\rho$.
\end{assu}

Remark that the convention on what is called the constant of the log-Sobolev inequality is different in \cite{GuillinWuZhang} and in the present paper, so that  $\rho$ here corresponds to $ 1/ (2\rho_{LS,m})$ in \cite{GuillinWuZhang}. This has no impact on the result, in both cases the one-particle conditional law  is required to satisfy a log-Sobolev inequality with a constant (in either sense) uniform in $N$ and in $x_{\neq 1}$. The same remark applies for the next result.

\begin{thm}[Theorem 8 of \cite{GuillinWuZhang}]\label{Thm:GWZ}
Under Assumptions \ref{Hyp:GWZ-1}, \ref{Hyp:GWZ-2} and \ref{Hyp:GWZ-3}, there exists $\eta>0$ such that $\pi_\infty^{(N)}$ (with $\beta=1$) satisfies a log-Sobolev with constant $\eta$ for all $N\in\mathbb N$.
\end{thm}

In \cite{GuillinWuZhang}, the one-particle conditional log-Sobolev inequality (i.e. Assumption \ref{Hyp:GWZ-3}) is proven under the assumption that the confinment is superconvex, meaning that $\nabla^2 V  \rightarrow +\infty$ at infinity. This is not compatible with the boundedness condition in Assumption \ref{Hyp:concret} but it is far from necessary.

In view of Lemma \ref{Lem:tensori} and Theorem \ref{Thm:GWZ}, Proposition \ref{Prop-logSob} thus follows from the following result:

\begin{lem}\label{Lem:AssuGWZ}
Assumption \ref{Hyp:concret} implies that Assumptions \ref{Hyp:GWZ-1}, \ref{Hyp:GWZ-2} and \ref{Hyp:GWZ-3} are satisfied by the potential $U_\beta = \beta U$ on $\R^{2d}$ and for the potential $H_\beta = \beta H$ on $\R^{4d}$ where $H$ is given by $H(x,y,x',y')=U(x,x')+(|y|^2+|y'|^2)/2$.
\end{lem}

\begin{proof}
We only detail the case of $U_\beta$, the case of $H_\beta$ is similar with $V(x)$ replaced by $V(x)+|y|^2/2$ and $W$ unchanged. In particular, $b_0(r)$ is the same in both cases (since the addition of the kinetic part is always non-positive).

Assumption \ref{Hyp:GWZ-1} is easily checked. Indeed, the existence of $c_1$ and $c_2$ follows from \eqref{Eq:hypV} applied with $y=0$ (and the fact $c_V>0$), and the integrability of $\exp(-\beta\po V(x)+V(y)+\lambda W(x,y)\pf )$ for all $\lambda>0$ follows from the fact $W$ is lower bounded.

Concerning Assumption \ref{Hyp:GWZ-2}, similarly to \cite[Remark 4]{GuillinWuZhang}, we see that, under the conditions \eqref{Eq:hypV} and \eqref{Eq:hypW}, the constant $c_L$ involved in Assumption \ref{Hyp:GWZ-2} is finite with
\[c_L \ \leqslant \ \frac{1}{\beta(c_V+c_W)} \exp\po \frac{\beta (c_V'+c_W') R}4\pf\,,\]
and thus
\[\gamma_0 \ \leqslant \  \frac{1}{(c_V+c_W)} \exp\po \frac{\beta (c_V'+c_W') R}4\pf \|\na_{x,y}^2 U\|_\infty \ < \  1\] 
where we used that $\beta<\beta_0$.

Finally, Assumption \ref{Hyp:concret} implies that $U=U_1+U_2$ where  $U_1$ is $\rho$-convex for some $\rho>0$ and $\|U_2\|_\infty <\infty$. Fix any $N\in\N_*$ and $x_{\neq 1}\in\R^{d(N-1)}$. Then $x_1\mapsto U_N(x)$ is the sum of a $\rho$-convex function and of a function bounded by $\|U_2\|_\infty <\infty$, so that the probability law with density proportional to  $x_1\mapsto \pi_\infty^{(N)}(x)$ satisfies a log-Sobolev with constant $e^{2\|U_2\|_\infty} /\rho$, by using Bakry-Emery's condition and Holley-Stroock perturbation argument (see \cite{BakryGentilLedoux}).

\end{proof}
Remark that we may consider slightly more general perturbation argument, namely Aida-Shigekawa \cite{AS} where $U_2$ could then be Lipschitzian but with Lipschitz constant less than $\rho/2$.

\subsection{First propagation of chaos estimates}
We will first establish uniform in time moment estimates both for the particles and non linear systems, which will come from classical Lyapunov arguments.

\begin{lem}\label{Lem:Moments}
Under Assumption \ref{Hyp:concret}, for all initial conditions $m_0\in \mathcal P_2(\R^{2d})$, there exists $K>0$ depending only on $U,\gamma,\sigma,m_0$  such that, if  $m_0^{(N)}=m_0^{\otimes N}$, then for all $N\in\N_*$ and $t\geqslant 0$,
\begin{eqnarray*}
\mathbb E\po |X_1(t)|^2 + |Y_1(t)|^2 \pf + \int_{\R^{2d}} \po |x|^2 + |y|^2\pf m_t(x,y)\dd x\dd y  & \leq & K\,.
\end{eqnarray*}
\end{lem}
\begin{proof}
Under Assumption \ref{Hyp:concret}, for $N\in\N_*$, $U_N$ satisfies, for all $x\in\R^{dN}$,
\begin{eqnarray}
x\cdot \na U_N(x) \ = \  \sum_{i=1}^N x_i \cdot \na_x U_N(x_i,x_j) & \geqslant & \sum_{i=1}^N    \po \po c_V+c_W\pf |x_i|^2 - \po c_V'+c_W'\pf |x_i| \pf \notag\\
& \geqslant & \frac{(c_V+c_W)}{2}|x|^2 - N \frac{(c_V'+c_W')^2}{2(c_V+c_W)}\label{eq:moments1}\,.
\end{eqnarray}
Moreover, from Lemma \ref{Lem:Uquadra}, for all $x\in\R^{dN}$,
\begin{equation}\label{eq:moments2}
\alpha_1 |x|^2 - \alpha_3 N \ \leqslant \  U_N(x) \ \leqslant  \ \alpha_2 |x|^2 + \alpha_3 N\,.
\end{equation}
From these estimates, the proof is then similar to the proof of \cite[Lemma 11]{MonmarcheVFP}, based on classical Lyapunov arguments. In the remaining of the proof, $c_i$ for $i\in\N$ denotes various positive constants that are independent from $N$ and $t$. The infinitesimal   the generator of $Z_N$ is
\begin{eqnarray}\label{EqDefiLN}
\mathcal L_N & = & y\cdot\na_x   - \po\na U_N(x)+\gamma y\pf \cdot \na_y   + \frac{\sigma^2}{2} \Delta_{y}\,.
\end{eqnarray}
For some $\varepsilon>0$, let
\[\tilde H(x,y) \ = \ U_N(x) + \frac12|y|^2 + \varepsilon x\cdot y\,.\]
Then, using \eqref{eq:moments1} and \eqref{eq:moments2}, for $\varepsilon$ small enough (and independent from $t$ and $N$),
\[\tilde H(x,y) \  \geqslant \ \frac{\alpha_1}{2}|x|^2 + \frac{1}{4}|y|^2 - \alpha_3 N\]
and
\begin{eqnarray*}
\mathcal L_N  \tilde H(x,y) &=& -(\gamma-\varepsilon) |y|^2 + \frac{\sigma^2}{2} dN - \varepsilon \gamma x\cdot y - \varepsilon x \cdot \na U_N(x)\\
  & \leqslant & -c_1 H  + c_2 N
\end{eqnarray*}
for some $c_1,c_2>0$. The Gr\"onwall Lemma yields
\begin{eqnarray*}
\mathbb E\po \tilde H\po Z_N(t)\pf\pf & \leq & \mathbb E\po \tilde H\po Z_N(0)\pf\pf + \frac{c_2 N}{c_1}\,. 
\end{eqnarray*}
Using the interchangeability of particles, 
\[\mathbb E\po |X_1(t)|^2+|Y_1(t)|^2\pf \ \leqslant \ \  \frac{c_3 }N \mathbb E\po \tilde H(Z_N(t))\pf + c_3  \ \leqslant \  c_4 \mathbb E\po |X_1(0)|^2+|Y_1(0)|^2\pf +  c_4  \]
for some $c_3,c_4>0$.

Similarly, for $\varepsilon >0$, denote
\[R_t \ = \ \int_{\R^{4d}} \po  U(x,x') + |y|^2 + \varepsilon x\cdot y\pf m_t(x,y)m_t(x',y') \dd x\dd y\dd x'\dd y'\,.\]
From  Assumption \ref{Hyp:concret} and Lemma \ref{Lem:Uquadra}, for $\varepsilon$ small enough,
\[R_t \ \geqslant \ c_5 \int_{\R^{2d}}\po |x|^2+|y|^2 \pf m_t(x,y)dxdy - c_6\]
for some $c_5,c_6>0$ and
\begin{eqnarray*}
\partial_t R_t &= & \int_{\R^{4d}}\po - (2\gamma-\varepsilon) |y|^2 + 2d  - \varepsilon \gamma x\cdot y - \varepsilon x \cdot \na_x U(x,x')\pf m_t(x,y)m_t(x',y') \dd x\dd y\dd x'\dd y'\\
& \leqslant & -c_7 R_t + c_8
\end{eqnarray*}
for some $c_7,c_8>0$. As a consequence,
\[\int_{\R^{2d}}\po |x|^2+|y|^2 \pf m_t(x,y)dxdy  \ \leqslant \  \frac{1}{c_5} \po R_0 + \frac{c_8}{c_7}\pf + \frac{c_6}{c5} \ \leqslant \ c_9\int_{\R^{2d}}\po |x|^2+|y|^2 \pf m_0(x,y)dxdy +c_9  \]
for some $c_9>0$, which concludes.
\end{proof}
Next proposition aims at providing a crude time dependent propagation of chaos estimate, however uniform in the number of particles.
\begin{prop}\label{prop:chaosW2}
Under Assumption \ref{Hyp:concret},  there exists $b$ (depending only on $U,\gamma,\sigma$) such that for all initial condition $m_0\in\mathcal P_2(\R^{2d})$ (and $m_0^{(N)} = m_0^{\otimes N}$), there exist $K>0$ (depending only on $U,\gamma,\sigma$  and $m_0$) such that for all $N\in\N_*$ and $t\geqslant 0$,
\begin{eqnarray*}
\mathcal W_2^2\po m_t^{\otimes N}, m_t^{(N)} \pf & \leqslant & K \po e^{bt}-1\pf\,.
\end{eqnarray*}
\end{prop}
\begin{proof}
This is a classical result, obtained with a parallel coupling of the system~\eqref{EqSystemparticul} with a system of independent non-linear particles. More precisely, consider a system $\overline Z_N = (\overline X_i,\overline Y_i)_{i\in\cco 1,N\ccf}$ with $\overline Z_N(0)=Z_N(0)$ and
\begin{eqnarray*}
\forall i\in \cco 1,N\ccf & &\left\{\begin{array}{rcl}
\dd \overline X_i & = & \overline Y_i \dd t\\
\dd \overline Y_i & = & - \gamma \overline Y_i \dd t - \int_{\R^{2d}}\nabla_x U(\overline X_i,x)m_t(x,y)\dd x\dd y \dd t + \sigma \dd B_i 
\end{array} \right. 
\end{eqnarray*} 
driven by the same Brownian motion as \eqref{EqSystemparticul}. The forces $\nabla_x U$ being Lipschitz, it is clear that there exist $b'>0$ such that
\begin{multline*}
\dd |Z_N-\overline Z_N|^2 \ \leqslant \ b'|Z_N-\overline Z_N|^2 \dd t \\
- \frac2N \sum_{i=1}^N \po Y_i-\overline Y_i\pf \sum_{j=1}^N \po \na_x W(X_i,X_j) - \int \na_x W(\overline X_i, u) m_t(u,v)\pf \dd t 
\end{multline*}
Decomposing the last term as
\[   \po \na_x W(X_i,X_j) -\na W_x(\overline X_i, \overline X_j)\pf   + \po \na W_x(\overline X_i,\overline X_j)  -  \int \na_x W(\overline X_i , u) m_t(u,v) \pf\,,\]
using that $\na_x W$ is Lipschitz, taking the expectation and using that particles are interchangeable, we obtain
\begin{multline*}
\partial_t \mathbb E\po |Z_N-\overline Z_N|^2\pf \ \leqslant \  b  \mathbb E \po |Z_N-\overline Z_N|^2\pf \\
+ \frac1N \mathbb E\po \left|\sum_{j=1}^N \po \na_x W(\overline X_1,\overline X_j)-\int \na_x W(\overline X_1,u)m_t(u,v)\pf\right|^2 \pf
\end{multline*}
 for some $b>0$. Finally, using that the $(\bar X_i,\bar Y_i)_{i\in\cco 1,N\ccf}$ are independent and distributed according to $m_t$,
 \begin{eqnarray*}
 \lefteqn{\mathbb E\po \left|\sum_{j=1}^N \po \na_x W(\overline X_1,\overline X_j)-\int \na_x W(\overline X_1,u)m_t(u,v)\pf\right|^2 \pf } \\
& = &   \mathbb E\po \sum_{j=1}^N\left| \na_x W(\overline X_1,\overline X_j)-\int \na_x W(\overline X_1,u)m_t(u,v)\right|^2 \pf    \\
&\leqslant&  N \|\na^2 W\|_\infty^2  \int_{\R^{2d}} |x|^2 m_t(x,y)\dd x \dd y\, .
 \end{eqnarray*}
 The moment estimates of Lemma~\ref{Lem:Moments} and   Grönwall's Lemma conclude.
\end{proof}
We will also need dome propagation of chaos estimates in entropy, that here will be inherited from estimates on Wasserstein distance.
\begin{prop}\label{prop:chaosEntropie}
Under Assumption \ref{Hyp:concret},  there exist $K$ (depending only on $U,\gamma,\sigma$  and $m_0$) such that for all $t\geqslant 0 $ and all $N\in\N_*$,
\[\mathcal H\po m_t^{(N)} \ | \ m_t^{\otimes N}\pf \ \leqslant \ K \po t + \sqrt{N} \int_0^t \mathcal W_2\po m_s^{\otimes N},m_s^{(N)}\pf\dd s \pf \,. \]
\end{prop}

\begin{proof}
We follow the idea of \cite[Lemma 3.15]{Malrieu} (see also \cite[Lemma 14]{MonmarcheVFP}), namely we compute the derivative of
\begin{eqnarray*}
F(t) & = &   \ent{m_t^{(N)}}{m_t^{\otimes N}}  .
\end{eqnarray*}
To do so, let $u_1 = m_t^{(N)}$, $u_2 = m_t^{\otimes N}$,
\[b_1(x,y) = \begin{pmatrix}
y \\ -\gamma y -  \na_x U_N(x)  
\end{pmatrix},\hspace{25pt}
b_2(x,y) = \begin{pmatrix}
y \\ -\gamma y -\na_x \overline U_{N}(x)
\end{pmatrix}
\]
with
\[\overline U_N(x) \ = \ \sum_{i=1}^N \int   U(x_i,v) m_t(v,w) \dd v \dd w,\]
 and $L_i f = - \na \cdot \po b_i  f \pf + \frac{\sigma^2}{2} \Delta_y f$ for $i=1,2$. With these notations, $\partial_t \po u_i\pf = L_i u_i$, and the dual in the Lebesgue sense of $L_i$ is $L_i' = b_i\cdot \na + \frac{\sigma^2}{2}  \Delta_y$. From the conservation of the mass of $u_1$, we get
\[  0 = \partial_t\po \int \frac{u_1}{u_2}u_2 \pf = \int \po L_1 u_1  - \frac{u_1}{u_2}L_2 u_2 +  L_2'\po\frac{u_1}{u_2}\pf u_2 \pf .  \]
Since $L_1'$ is a diffusion operator with carr\'e  du champ operator $\Gamma f = \frac{\sigma^2}{2}  |\na_y f|^2$ (see \cite[p.20 \& 42]{BakryGentilLedoux} for the definitions),
\[ u_1 L_1' \ln\po \frac{u_1}{u_2} \pf = u_1 \frac{L_1'\po\frac{u_1}{u_2}\pf}{\frac{u_1}{u_2}} - u_1\frac{\Gamma\po \frac{u_1}{u_2}\pf}{\po \frac{u_1}{u_2}\pf^2} = u_2 L_1'\po\frac{u_1}{u_2}\pf - u_1 \Gamma\po \ln \frac{u_1}{u_2} \pf.\]
Using both these relations,
\begin{eqnarray*}
\partial_t \po \int \ln\po\frac{u_1}{u_2}\pf u_1\pf & = & \int \po \frac{L_1 u_1}{u_1} - \frac{L_2 u_2}{u_2} + L_1' \ln\po\frac{  u_1}{u_2}\pf\pf u_1\\
& = & \int -\Gamma\po \ln \frac{u_1}{u_2} \pf u_1 + u_2 L_1'\po\frac{u_1}{u_2}\pf - u_2 L_2'\po\frac{u_1}{u_2}\pf \\
& = &   \int -\Gamma\po \ln \frac{u_1}{u_2} \pf u_1 +  (b_1-b_2)\cdot \na \ln \po\frac{u_1}{u_2}\pf u_1.
\end{eqnarray*}
Applying Young's Inequality, we get
\begin{eqnarray*}
F'(t) & \leq & \frac1{2\sigma^2} \int \left| \na  U_N(x)  - \na  \overline U_N(x) \right|^2 m_t^{(N)}\\
& = & \frac{N}{2\sigma^2} \mathbb E\po \left|\frac1N \sum_{j=1}^N \na_x  W(X_1,X_j) - \int \na_x W(X_1,v) m_t(v,w)\right|^2\pf
\end{eqnarray*}
by interchangeability. Developing the square of the sum, the $N$ diagonal terms are bounded by
\begin{eqnarray*}
  \frac1{N^2}\| \nabla ^2 W\|_\infty^2 \po \mathbb  E\po  |X_j|^2\pf + \int |v|^2 m_t(v,w) \pf &\leq & \frac{K}{N^2}
\end{eqnarray*}
for some $K>0$ where we used  Lemma \ref{Lem:Moments}. For the extra-diagonal terms, we consider an optimal coupling $(\overline Z_N,Z_N)$ of  $m_t^{\otimes N}$ and $m_t^{(N)}$ in the sense that
\begin{eqnarray*}
\mathbb E\po \left|\overline Z_N(t)-Z_N(t)\right|^2\pf &=& \mathcal W_2^2\po m_t^{\otimes N},m_t^{(N)}\pf
\end{eqnarray*} 
and write, for $j\neq k$,
\begin{eqnarray*}
& & \po \na  W(X_1,X_j) - \int \na  W(X_1,v) m_t \pf  \po  \na  W(X_1,X_k) - \int \na  W(X_1,v) m_t \pf\\
 & = & \po \na  W(X_1,X_j) - \na  W(X_1,\overline X_j)\pf\po  \na  W(X_1,X_k) - \int \na  W(X_1,v) m_t \pf\\
 & & + \po \na  W(X_1,\overline X_j) - \int \na  W(X_1,v) m_t \pf\po  \na  W(X_1,X_k) - \na  W(X_1,\overline X_k)\pf\\
 & & + \po \na  W(X_1,\overline X_j) - \int \na  W(X_1,v) m_t \pf \po \na  W(X_1,\overline X_k) - \int \na  W(X_1,v) m_t \pf.
\end{eqnarray*}
The $\overline X_i$'s being independent with law the first marginal of $m_t$, the expectation of the third term vanishes, while the expectations of the two other terms is bounded by the Cauchy-Schwarz inequality and interchangeability  by
\[  \| \nabla ^2 W\|_\infty^2 \sqrt{\mathbb  E\po  |X_1 - \overline{X}_1|^2\pf \po \mathbb  E\po  |X_1|^2\pf + \mathbb  E\po  |\overline{X}_1|^2\pf  \pf} \ \leqslant \ \frac{K}{\sqrt N} \mathcal W_2\po m_t^{\otimes N},m_t^{(N)}\pf\]
for some $K>0$ where we used again interchangeability and Lemma~\ref{Lem:Moments} for the second inequality. As a conclusion, we have obtained that for all $t\geqslant 0$
\begin{eqnarray*}
F'(t) & \leqslant & K + K \sqrt{ N} \mathcal W_2\po m_t^{\otimes N},m_t^{(N)}\pf 
\end{eqnarray*}
for some $K>0$ independent from $t$ and $N$, and the claims follows from the fact $F(0)=0$.
\end{proof}

\begin{lem}\label{Lem:propchaos-n}
For all $t\geqslant 0 $,   $N\in\N_*$ and $n\in \cco 1,N\ccf$,
\begin{eqnarray*}
\mathcal W_2^2\po m_t^{\otimes n}, m_t^{(n,N)} \pf & \leqslant & \frac{n}N \mathcal W_2^2\po m_t^{\otimes N}, m_t^{(N)} \pf\\
\mathcal H \po m_t^{(n,N)}\ |\ m_t^{\otimes n}\pf & \leqslant & \frac{1}{\lfloor N/n\rfloor} \mathcal H \po m_t^{(N)}\ |\ m_t^{\otimes N}\pf \,.
\end{eqnarray*}
\end{lem}
\begin{proof}
Let $Z_N=((X_1,Y_1),\dots,(X_N,Y_N))$ and $\overline Z_N = ((\overline X_1,\overline Y_1),\dots,(\overline X_N,\overline Y_N))$ be a $\mathcal W_2$-optimal coupling of $m_t^{(N)}$ and $m_t^{\otimes N}$, i.e. be such that $Z_N\sim m_t^{(N)}$, $\overline Z_N\sim m_t^{\otimes N}$ and 
\[\mathcal W_2^2 \po m_t^{(N)},m_t^{\otimes N}\pf \ = \ \mathbb E\po |Z_N-\overline Z_N|^2 \pf\,. \]
Then $((X_1,Y_1),\dots,(X_n,Y_n))$  and  $((\overline X_1,\overline Y_1),\dots,(\overline X_n,\overline Y_n))$ are a coupling of   $m_t^{(n,N)}$ and $m_t^{\otimes n}$ and, by exchangeability,
\[\mathcal W_2^2 \po  m_t^{(n,N)}, m_t^{\otimes n}\pf \ \leqslant \ \sum_{i=1}^n \mathbb E\po |X_i-\overline X_i|^2 + |Y_i-\overline Y_i|^2\pf \ = \ \frac{n}N \mathcal W_2^2 \po  m_t^{(N)}, m_t^{\otimes N}\pf\,.\]
The second claim follows  from the Csisz{\'a}r's inequality which is \cite[Inequality (2.10)]{Csiszar} for $n=1$. Let us establish it for any $n\in\cco1,N\ccf$. Set $k=\lfloor N/n\rfloor$ and $s=N-kn$.
\begin{eqnarray*}
\int_{\R^{2dN}} m_t^{(N)} \ln \po \frac{m_t^{(N)}}{m_t^{\otimes N}}\pf \dd z  & = & \int_{\R^{2dN}} m_t^{(N)} \ln \po \frac{m_t^{(N)}}{\po m_t^{(n,N)}\pf^{\otimes k}\otimes m_t^{\otimes s}}\pf \dd z \\
& & +\  \int_{\R^{2dN}} m_t^{(N)} \ln \po \frac{\po m_t^{(n,N)}\pf^{\otimes k}\otimes m_t^{\otimes s}}{m_t^{\otimes N}}\pf \dd z \\
& \geqslant & k \int_{\R^{2dn}} m_t^{(n,N)} \ln \po \frac{  m_t^{(n,N)}}{m_t^{\otimes n}}\pf \dd z
\end{eqnarray*}
where we used that the first term is positive (as a relative entropy) and the interchangeability of $m_t^{(N)}$.
\end{proof}

\subsection{Long-time convergence}

The proof of Theorem~\ref{TheoLine} is based on the following  quantitative results of hypocoercivity for diffusion processes.

\begin{thm}[from Theorem 10 of  \cite{MonmarcheGamma}]\label{Theo10Gamma}
Consider a diffusion generator $L$  on H\"ormander form
\[L = B_0 + \sum_{i=1}^d B_i^2\]
where the $B_j$'s are derivation operators. Suppose there exist $N_c\in \mathbb N$ and $\lambda,\Lambda,m,\rho,K>0$ such that for $i\in\llbracket 0,N_c+1\rrbracket$ there exist smooth derivation operators $C_i$ and $R_i$ and a scalar field $Z_i$  satisfying:
\begin{enumerate}
\item $C_{N_c+1} = 0$,\ and \  $[B_0,C_i] = Z_{i+1} C_{i+1} + R_{i+1}$ \  for all $i\in\llbracket 0,N_c\rrbracket$, where $[A,B]=AB-BA$ stands for the Poisson bracket of two operators,
\item $[B_j,C_i] = 0$ \ for all $i\in\llbracket 0,N_c\rrbracket$, $j\in \llbracket 1,d\rrbracket$, 
\item  $\lambda \leq  Z_i  \leq \Lambda$ \ for all $i\in\llbracket 0,N_c\rrbracket$,
\item $|C_0 f|^2 \leq m \underset{j\geq 1}\sum |B_j f|^2$ and $|R_i f|^2 \leq m \underset{j<i}\sum |C_j f|^2$ for all $i\in\llbracket 0,N_c+1\rrbracket$ and smooth Lipschitz $f$.
\item $\underset{i\geq 0}\sum | C_i f|^2 \geq \rho |\na f|^2$.
\end{enumerate}
Suppose moreover that there exists a probability measure $\mu$ which is invariant for $e^{tL}$ and satisfies a log-Sobolev inequality with constant $\eta$.

\bigskip

Then   for all $t>0$ and for all $f>0$ with $\int f\dd \mu = 1$,
\begin{eqnarray}\label{EqTheorCoerc}
 \int \po e^{tL}f\pf\ln \po e^{tL}f\pf\dd \mu & \leq &  e^{-\kappa t(1-e^{-t})^{2N_c}}  \int f\ln f\dd \mu
\end{eqnarray}
with
\begin{eqnarray*}
\kappa &=& \frac{\rho}{\eta} \po \frac{100}{\lambda}\po N_c^2 + \frac{\Lambda^2}{\lambda}+m\pf \pf^{-20 N_c^2}.
\end{eqnarray*}
\end{thm}

\begin{proof}[Proof of Theorem \ref{TheoLine}]
From Theorem~\ref{Theo10Gamma}, the uniform log-Sobolev inequality given by Proposition \ref{Prop-logSob} and the bound on $\|\na^2 U_N\|_\infty$ that is uniform in $N$, the proof of \eqref{Eq:EntropieLine}  is similar to the proof of \cite[Theorem 1]{MonmarcheVFP}. The generator \eqref{EqDefiLN} is on H\"ormander form 
\[B_0 + \sum_{i=1}^N \sum_{j=1}^d B_{i,j} \]
with, writing $y_i = \po y_i^{(1)},\dots,y_i^{(d)}\pf\in\mathbb R^d$,
\begin{eqnarray*}
B_0 &=& -y\cdot\na_x   + \po\na U_N(x)-\gamma y\pf \cdot \na_y   \\
B_{i,j} &=& \frac{\sigma}{\sqrt 2} \partial_{y_i^{(j)}}.
\end{eqnarray*}
 Since
 \[ [B_0,\na_y] =[L_N,\na_y] = \na_x + \gamma\na_y, \hspace{35pt}[B_0,\na_x] =[L_N,\na_x] = - \na_x^2 U_N \na_y,\]
 Theorem \ref{Theo10Gamma} applies with 
\[C_0 = \na_y,\qquad C_1 = \na_x,\qquad R_1 = \gamma \nabla_y,\qquad R_2 = - \na^2 U_N \na_y,\]
  \[Z_1=Z_2=N_c=\lambda=\Lambda=\rho=1,\qquad m= \frac2{\sigma^2 } +  \gamma^2 +\po\|\na^2 V\|_\infty+2\|\na^2 W\|_\infty\pf^2\]
  and $\eta$ given by  Proposition \ref{Prop-logSob}. This gives \eqref{Eq:EntropieLine}.
   
   \bigskip

The second part of Theorem \ref{TheoLine}, namely \eqref{Eq:RegulEntropie}, follows from  \cite{GuillinWang} whose useful (for us) results are gathered in the following proposition.
\begin{prop}[From Corollary 4.7 in \cite{GuillinWang}]
Consider the stochastic differential equation on $\mathbb{R}^m\times\mathbb{R}^d$
$$dX_t=AY_tdt,\qquad dY_t=dB_t+Z_t(X_t,Y_t)dt$$
with initial conditions $(X_0,Y0)=(x,y)$, and associated semigroup $P_t$. Assume
$$|\nabla^xZ(x,y)|\le K_1, \, |\nabla^y Z(x,y)|\le K_2,\qquad (x,y)\in\mathbb{R}^m\times\mathbb{R}^d.$$
Suppose also that $P_t$ has an invariant probability measure $\mu$ and let $P_t^*$ be the adjoint of $P_t$ in $L^2(\mu)$. then for every $t>0$, function $f\ge0$ with $\mu(f)=1$
$$\mu(P_t^*f\log P_t^*f)\le \frac{C}{(1\wedge t)^3} W_2^2(f\mu,\mu)$$
where $C$ only depends on $K_1$ and $K_2$.
\end{prop}

 It is a regularization result, namely a control in small time  of the entropy along the flow of the particles system by the initial Wasserstein distance. Most importantly for us, this regularization has to be independent of the number of particles. Let us check that indeed the constant $C$ obtained does not depend on $N$. Applied to our case, the notations read $A=I_{dN}$ and $Z(x,y) = -\na U_N(x) - \gamma y$. Under Assumption~\ref{Hyp:concret}, the Jacobian matrix of this $Z$ is bounded uniformly in $N$, which means that $K_1$ and $K_2$ do not depend on $N$, and thus neither does $C$, which concludes.

Finally, at least for $t\geqslant 1$, \eqref{Eq:W2Line} is a straightforward consequence of the two previous claims of Theorem \ref{TheoLine} and of the Talagrand $T_2$ inequality implied by the log-Sobolev inequality given by Proposition~\ref{Prop-logSob}. Indeed, for $t\geqslant 1$,
\begin{eqnarray*}
\mathcal W_2^2 \po m_t^{(N)},m_\infty^{(N)}\pf &\leqslant& \eta \mathcal H\po m_t^{(N)}\ |\ m_\infty^{(N)}\pf \\
&\leqslant &\eta Ce^{-\chi(t-1)} \mathcal H\po m_1^{(N)}\ |\ m_\infty^{(N)}\pf\\
& \leqslant & \eta C^2 e^{-\chi t} W_2^2\po m_0^{(N)},m_\infty^{(N)}\pf\,.
\end{eqnarray*}
For $t\in[0,1]$, we simply consider two solutions $Z_N,\tilde Z_N$  of \eqref{EqSystemparticul} driven by the same Brownian motion but with two different initial condition. More precisely, we suppose that $(Z_N(0),\tilde Z_N(0))$ is an $\mathcal W_2$-optimal coupling of $m_0^{(N)}$ and $m_\infty^{(N)}$, so that
\[\mathbb E \po |Z_N(0)-\tilde Z_N(0)|^2\pf \  = \ \mathcal W_2^2\po m_0^{(N)},m_\infty^{(N)}\pf\,.\]
Since $\|\na^2 U_N\|_\infty$ is bounded uniformly in $N$, we immediatly get that 
\[\dd |Z_N(t)-\tilde Z_N(t)|^2 \ \leqslant \ b |Z_N(t)-\tilde Z_N(t)|^2\dd t \]
for some $b>0$ that does not depend on $N$. Conclusion follows from
\[\mathcal W_2^2\po m_t^{(N)},m_\infty^{(N)}\pf \ \leqslant \ \mathbb E \po |Z_N(t)-\tilde Z_N(t)|^2\pf \ \leqslant \ e^{bt} \mathbb E\po |Z_N(0)-\tilde Z_N(0)|^2\pf \,. \]
\end{proof}

%

Let us now transfer the results obtained on the particles system to the nonlinear equation.

\begin{proof}[Proof of Theorem \ref{TheoNonLine}]

In this proof, we use repeatedly results from \cite{GuillinWuZhang} but applied to the potential $H_\beta$ defined in Lemma \ref{Lem:AssuGWZ}. It is possible to do so since, according to Lemma~\ref{Lem:AssuGWZ}, this potential satisfies the assumptions of \cite{GuillinWuZhang} (in particular the condition $c_L\|\na^2_{x,y} W\|_\infty<1$).

The fact that $E_f$ admits a unique minimizer $m_\infty$ over $\mathcal P(\R^{2d})$ is proven in \cite[Lemma 21]{GuillinWuZhang}. Moreover, as established in the proof of \cite[Theorem 10]{GuillinWuZhang}, $\mu_\infty^{(1,N)}$ weakly converges to $m_\infty$ and for all $\nu\in\mathcal P_2(\R^{2d})$,
\[\mathcal W_2^2 (\nu,m_\infty) \ \leqslant \ \liminf_{N\rightarrow+\infty} \frac1N \mathcal W_2^2 \po \nu^{\otimes N},m_\infty^{(N)}\pf\,.\]
Moreover, according to \cite[Lemma 17]{GuillinWuZhang}, for all $\nu\in\mathcal P(\R^{2d})$ such that $\mathcal H(\nu|\alpha)<+\infty$,
\begin{eqnarray}\label{Eq:Ninfty1}
\frac1N \mathcal H\po \nu^{\otimes N}|m_\infty^{(N)}\pf & \underset{{N\rightarrow +\infty}}\longrightarrow & \mathcal H_W(\nu)\,.
\end{eqnarray}
Applied with $\nu = m_\infty$ and combined with the Talagrand's Inequality satisfied by $m_\infty^{(N)}$,
\[\frac1N  \mathcal W_2^2\po m_\infty^{\otimes N},m_\infty^{(N)}\pf \  \leqslant \ \frac\eta N \mathcal H \po m_t^{\otimes N}\ | \ m_\infty^{(N)}\pf \ \underset{N\rightarrow +\infty}\longrightarrow \ \mathcal H_W(m_\infty) \ = \ 0\,.\]
In particular, dividing
\[ \mathcal W_2\po m_0^{\otimes N},m_\infty^{(N)}\pf \ \leqslant \ \sqrt{N}\mathcal W_2\po m_0,m_\infty \pf + \mathcal W_2\po m_\infty^{\otimes N},m_\infty^{(N)}\pf \]
by $\sqrt N$ and letting $N\rightarrow +\infty$ we get that 
\begin{eqnarray}\label{Eq:Ninfty2}
\limsup_{N\rightarrow +\infty} \frac1{\sqrt N} \mathcal W_2\po m_0^{\otimes N},m_\infty^{(N)}\pf  & \leqslant &  \mathcal W_2\po m_0,m_\infty \pf\,.
\end{eqnarray}
Together with Theorem~\ref{TheoLine} and Proposition~\ref{prop:chaosW2}, for all $t\geqslant 0$,
\begin{eqnarray*}
\mathcal W_2 (m_t,m_\infty) & \leqslant & \limsup_{N\rightarrow+\infty} \po \mathcal W_2\po m_t,m_t^{(1,N)}\pf + \frac1{\sqrt N} \mathcal W_2\po m_t^{(N)},m_\infty^{(N)}\pf  + \mathcal W_2\po m_\infty^{(1,N)},m_\infty\pf\pf\\
& \leqslant & C e^{-\chi t} \mathcal W_2\po m_0,m_\infty \pf\,.
\end{eqnarray*}

Similarly, following the proof of \cite[Theorem 10]{GuillinWuZhang} we see that
\begin{eqnarray}\label{Eq:Ninfty3}
\mathcal H_W\po m_t\pf & \leqslant & \liminf_{N\rightarrow+\infty} \frac1N \mathcal H\po m_t^{(N)}\ |\ m_\infty^{(N)}\pf\,.
\end{eqnarray}
The proof of \eqref{Eq:EntropieNonLin} and \eqref{Eq:RegulNonLin} follows then from dividing \eqref{Eq:EntropieLine} and \eqref{Eq:RegulEntropie} by $N$ and letting $N\rightarrow +\infty$ thanks to \eqref{Eq:Ninfty1}, \eqref{Eq:Ninfty2} and \eqref{Eq:Ninfty3}.

\end{proof}


\subsection{Proofs of the corollaries}
We first need some preliminary lemmas. The first ones gives a control of the propagation of chaos at the level of the invariant measure (so at infinite time).
\begin{lem}\label{lem:propchaosEq}
Under Assumption \ref{Hyp:concret}, there exists $K>0$ such that for all $N\in\N_*$,
\[\mathcal W_2 \po m_\infty^{\otimes N},m_\infty^{(N)}\pf \ \leqslant \ K\,.\]
\end{lem}

\begin{proof}
For all $N\in \N_*$ and $t\geqslant 0$,
\[\mathcal W_2 \po m_\infty^{\otimes N}, m_\infty^{(N)}\pf \ \leqslant \ \mathcal W_2 \po m_\infty^{\otimes N}, m_t^{(N)}\pf + \mathcal W_2 \po m_t^{( N)}, m_\infty^{(N)}\pf\,. \]
Applied in the case $m_0^{(N)} = m_\infty^{\otimes N}$ together with Proposition \ref{prop:chaosW2} and Theorem \ref{TheoLine}, this yields
\[\mathcal W_2 \po m_\infty^{\otimes N}, m_\infty^{(N)}\pf \ \leqslant \ K e^{bt} + C e^{-\chi t} \mathcal W_2 \po m_\infty^{\otimes N}, m_\infty^{(N)}\pf\,. \]
In particular, for $t=\ln(2C)/\chi$, we get
\[\mathcal W_2 \po m_\infty^{\otimes N}, m_\infty^{(N)}\pf \ \leqslant \ 2K(2C)^{b/\chi}\,. \]
\end{proof}

\begin{lem}\label{LemCoupleEmpirique}
Let $\nu_1$ and $\nu_2$ be probability laws on $\R^{dN}=\po\R^d\pf^{N}$ which are fixed by any permutation of the  $d$-dimensional coordinates (in other words, if  $(A_{i})_{i\in\cco 1,N\ccf}$ is of law $\nu$, the $A_i$'s are  interchangeable). Let $(A,B)=(A_i,B_i)_{i\in\cco 1,N\ccf}$ be a coupling of $\nu_1$ and $\nu_2$ such that
\[\mathbb E\po |A-B|^2\pf = \mathcal W^2_2(\nu_1,\nu_2).\]
Then
\begin{eqnarray*}
\mathbb E\po \mathcal W^2_2 \po \frac1N\sum_{i=1}^N \delta_{A_i}, \frac1N\sum_{i=1}^N \delta_{B_i}\pf \pf & \leq & \frac{1}{N} \mathcal W^2_2(\nu_1,\nu_2).
\end{eqnarray*} 
\end{lem}
\begin{proof}
Let $I$ be uniformly distributed on $\cco 1,N\ccf$. Then $(A_I,B_I)$ is a coupling of $ \frac1N\sum \delta_{A_i}$ and $ \frac1N\sum \delta_{B_i}$, hence
\begin{eqnarray*}
\mathbb E\po \mathcal W^2_2 \po \frac1N\sum_{i=1}^N \delta_{A_i}, \frac1N\sum_{i=1}^N \delta_{B_i}\pf \pf & \leq & \mathbb{E}\po |A_I-B_I|^2\pf\\
& = & \frac1N \mathbb E\po |A-B|^2\pf.
\end{eqnarray*} 
\end{proof}

\begin{proof}[Proof of Corollary \ref{CorPW2}]
Let $(Z_N,\tilde Z_N)$ be a $\mathcal W_2$-optimal coupling of $m_t^{(N)}$ and $m_\infty^{\otimes N}$ and $\tilde M_t^N$ be the empirical distribution of $\tilde Z_N$. Then, using Lemma \ref{LemCoupleEmpirique}, we bound
\begin{eqnarray*}
\mathbb E\po \mathcal W_2^2 \po M_t^N ,m_\infty\pf\pf & \leqslant & 2\mathbb E\po \mathcal W_2^2 \po M_t^N ,\tilde M_t^N \pf \pf + 2\mathbb E\po \mathcal W_2^2 \po \tilde M_t^N ,m_\infty\pf \pf \\
& \leqslant & \frac{2}N \mathcal W_2^2 \po m_t^{(N)} ,m_\infty^{\otimes N}\pf   + 2\mathbb E\po \mathcal W_2^2 \po \tilde M_t^N ,m_\infty\pf \pf 
\end{eqnarray*}
From \cite[Theorem 1]{GuillinFournier},  the second term is bounded by $R a(N)$ for some $R$ independent from $N$, and we bound the first one using Lemma \ref{lem:propchaosEq} and Theorem \ref{TheoLine} as 
 \begin{eqnarray*}
 \mathcal W_2 \po m_t^{(N)} ,m_\infty^{\otimes N}\pf  & \leqslant &  \mathcal W_2 \po m_t^{(N)} ,m_\infty^{(N)}\pf + \mathcal W_2 \po m_\infty^{(N)} ,m_\infty^{\otimes N}\pf \\
 & \leqslant &  Ce^{-\chi t} \mathcal W_2 \po m_0^{(N)} ,m_\infty^{(N)}\pf + K \\
  & \leqslant &  Ce^{-\chi t} \po \mathcal W_2 \po m_0^{\otimes N } ,m_\infty^{\otimes N}\pf + \mathcal W_2 \po m_\infty^{( N) } ,m_\infty^{\otimes N}\pf \pf + K\\
   & \leqslant &  Ce^{-\chi t} \sqrt{N} W_2 \po m_0 ,m_\infty \pf + K(1+C)\\
   & \leqslant & K'(\sqrt N e^{-\chi t} + 1)
 \end{eqnarray*}
 for some $K'$ independent from $N$ and $t$. We have thus obtained 
 \begin{eqnarray*}
\mathbb E\po \mathcal W_2^2 \po M_t^N ,m_\infty\pf\pf & \leqslant &  4(K')^2 \po e^{-2\chi t} + \frac1N\pf + R a(N)\,,
\end{eqnarray*}
and conclusion follows from the fact $1/N$ is neglictible with respect to $a(N)$ as $N\rightarrow +\infty$.
\end{proof}

\begin{proof}[Proof of Corollary \ref{CorChaosUniforme}]
Combining  Proposition \ref{prop:chaosW2} and Lemma \ref{Lem:propchaos-n}, 
\[\mathcal W_2^2\po m_t^{\otimes n},m_t^{(n,N)}\pf \ \leqslant \  \frac{Kn e^{bt}}{N}\]
Besides, combining Theorems \ref{TheoLine} and \ref{TheoNonLine} and Lemma \ref{lem:propchaosEq},
\begin{eqnarray*}
\mathcal W_2\po m_t^{\otimes N},m_t^{(N)} \pf &\leqslant & \mathcal W_2\po m_t^{\otimes N},m_\infty^{\otimes N} \pf + \mathcal W_2\po m_\infty^{\otimes N},m_\infty^{(N)} \pf + \mathcal W_2\po m_\infty^{(N)},m_t^{(N)} \pf \\
& \leqslant& C e^{-\chi t} \po \sqrt N \mathcal W_2(m_0,m_\infty) +   \mathcal W_2\po m_\infty^{(N)},m_0^{(N)} \pf \pf + K\\
& \leqslant& C e^{-\chi t} \po 2 \sqrt N \mathcal W_2(m_0,m_\infty) +   \mathcal W_2\po m_\infty^{(N)},m_\infty^{\otimes N} \pf \pf + K\\
& \leqslant & K' \po \sqrt N e^{-\chi t} + 1\pf
\end{eqnarray*}
for some $K'$ independent from $N$ nor $t\geqslant 0$. Again with Lemma \ref{Lem:propchaos-n}, we have thus obtained that there exists $K''$ independent from $N$ and $t$ such that
\[\mathcal W_2^2\po m_t^{\otimes n},m_t^{(n,N)}\pf \ \leqslant \  K'' n \po \frac{e^{bt}}{N} \wedge \frac{N}{ e^{2\chi t}}\pf \,.\]
Distinguishing the cases $t\leqslant \ln(N)/(2b)$ and $t\geqslant \ln(N)/(2b)$ concludes the proof for the $\mathcal W_2$ distance.

The case of the total variation distance is similar. First, from Pinsker's and Csisz{\'a}r's inequalities, considering the initial condition $m_0^{(N)}=m_\infty^{\otimes N}$, we get for all $t\geqslant 1$
\begin{eqnarray*}
\| m_\infty^{\otimes n} - m_\infty^{(n,N)} \|_{TV}^2 & \leqslant & 2\| m_\infty^{\otimes n} - m_t^{(n,N)} \|_{TV}^2+2\| m_t^{(n,N)} - m_\infty^{(n,N)} \|_{TV}^2 \\
& \leqslant & \frac {8n}N \mathcal H\po m_t^{(N)}|m_t^{\otimes N}\pf + 4\mathcal H\po m_t^{(N)}|m_\infty^{(N)}\pf \\
& \leqslant & \frac {8n}N K' e^{bt} \sqrt{N} + 4 C^2 e^{-\chi (t-1)} \mathcal W_2^2 \po m_0^{(N)},m_\infty^{(N)}\pf
\end{eqnarray*}
for some $K'$, where we combined Propositions \ref{prop:chaosW2} and \ref{prop:chaosEntropie} for the first term and used Theorem \ref{TheoLine} for the second one. Together with Lemma \ref{lem:propchaosEq}, we have obtained that for some $K''$ independent from $N,t,n$, 
\begin{eqnarray*}
\| m_\infty^{\otimes n} - m_\infty^{(n,N)} \|_{TV}^2 & \leqslant & K'' \po  \frac {n}{\sqrt{N}}  e^{bt}   +    e^{-\chi t}  \pf \ \leqslant \ n K'' \po  \frac {1}{\sqrt{N}}  e^{bt}   +    e^{-\chi t}  \pf \ \leqslant \ \frac{K'''n}{N^{\kappa}}
\end{eqnarray*}
for some $\kappa,K'''>0$ when $t= 1+\ln(N)/(4b)$.

Now, considering any initial condition $m_0\in\mathcal P_2(\R^d)$,
\begin{eqnarray*}
\| m_t ^{\otimes n} - m_t^{(n,N)} \|_{TV}^2 & \leqslant & 3\| m_t ^{\otimes n} - m_\infty^{\otimes n} \|_{TV}^2 + 3\| m_\infty ^{\otimes n} - m_\infty^{(n,N)} \|_{TV}^2 + 3\| m_\infty ^{(n,N)} - m_t^{(n,N)} \|_{TV}^2 \\
& \leqslant & 6 \mathcal H \po m_t^{\otimes n}|m_\infty^{\otimes n}\pf  + 3K'''n N^{-\kappa} + 6 \mathcal H\po m_t^{(N)}|m_\infty^{(N)}\pf \\
& \leqslant & 6 C^2 e^{-\chi (t-1)} \po \mathcal W_2^2\po m_0^{\otimes N},m_\infty^{\otimes N}\pf +  \mathcal W_2^2\po m_0^{(N)},m_\infty^{( N)}\pf \pf+  3K'''n N^{-\kappa}
\end{eqnarray*}
for $t\geqslant 1$, so that
\begin{eqnarray*}
\| m_t ^{\otimes n} - m_t^{(n,N)} \|_{TV}^2 & \leqslant & K\po Ne^{-\chi t} +   n N^{-\kappa} \pf \ \leqslant Kn\po Ne^{-\chi t} + N^{-\kappa}\pf
\end{eqnarray*}
for some $K$. Besides, from Propositions \ref{prop:chaosW2} and \ref{prop:chaosEntropie} and Lemma \ref{Lem:propchaos-n},
\begin{eqnarray*}
\| m_t ^{\otimes n} - m_t^{(n,N)} \|_{TV}^2 & \leqslant & \frac{4n}{N} \mathcal H \po m_t^{(N)}|m_t^{\otimes N}\pf \ \leqslant \ \frac{Kn} {\sqrt N} e^{bt}
\end{eqnarray*}
for some $K$, and conclusion follows again by distinguishing the cases $t\geqslant 1+\ln(N)/(4b)$ and $t\leqslant 1+\ln(N)/(4b)$.
\end{proof}

\section*{Acknowledgements}

A. Guillin and P. Monmarch\'e acknowledge financial support from  the French ANR grant EFI (Entropy, flows, inequalities,  ANR-17-CE40-0030).

\bibliographystyle{plain}
\bibliography{biblio}

\end{document}